\documentclass[font=11pt]{amsart}
\usepackage{amstext,amsmath,amssymb,amsthm, bm}
\usepackage{latexsym}
\usepackage{exscale}
\usepackage{tikz}
\usepackage{xcolor}
\usepackage[margin=1.25in]{geometry}
\usepackage{hyperref}
\usepackage[capitalize]{cleveref}
\usepackage{enumerate}

\numberwithin{equation}{section}

\newcommand\Z{\mathbb Z}
\newcommand\N{\mathbb N}
\renewcommand\t {\tilde}

\def\R{\mathbb R}

\def\({\Big(}
\def\){\Big)}

 \newcommand\norm[1]{\left\lVert#1\right\rVert}
\renewcommand{\l}{\left\langle}
\renewcommand{\r}{\right\rangle}
\renewcommand{\t}{\tilde}
\newcommand\rect{\operatorname{rect}}

\def\I{I_{q,A}}
\def\G{\mathcal G}
\def\B{\mathcal B}
\def\C{\mathbb C}
\def\I{\mathcal I}

\def\F{\mathcal F}

\newtheorem{theorem}{Theorem}[section]
\newtheorem{lemma}[theorem]{Lemma}
\newtheorem{proposition}[theorem]{Proposition}
\newtheorem{corollary}[theorem]{Corollary}

\theoremstyle{definition}

\theoremstyle{remark}

\numberwithin{equation}{section}

\newcommand{\dsize}{\displaystyle}

\begin{document}

\title{{\large  Frequency-dependent  stability   for  Gabor frames  }}

  \author{Laura De Carli}
  \address{Dept. of Mathematics, Florida International University,   Miami, FL 33199, USA.}
  \email{decarlil@fiu.edu}
 
  \author{Luis Rodriguez}
 \address{Dept. of Mathematics, Florida International University,   Miami, FL 33199, USA.}
 
 \email{rodriguezl@fiu.edu}
 \author{Pierluigi Vellucci}
 \address{Dept. of Economics, Roma Tre University, via Silvio D'Amico 77, 00145 Rome, Italy.}
 \email{pierluigi.vellucci@uniroma3.it}
\begin{abstract}
	We prove new stability results for a class of regular   Gabor frames  $G(h; a,b )$ subject to frequency-dependent  timing jitter  under various conditions on  the window function $h$.

\end{abstract}

\maketitle






\bigskip

\section{Introduction}
\label{sec:intro}

 A {\it Gabor system} in $L^2(\R)$ is a collection of functions   $\G=\{e^{2\pi i b_n x}g(x-a_k)\}_{n,k\in\Z}$ , where  $g$ (the {\it window function}) is  in $L^2(\R)$ and  $\{ a_k,\ b_n\}_{k, n\in\Z}\subset \R$.

If    {$(a_k, b_n)=(ak, bn)$} for some $a, b>0$ and all $n,\, k\in\Z$,   we say that $\G$ is {\it regular} and we let  {$\G(g; a,b)=\{e^{2\pi  b i nx} g ( x-ak)\}_{n,k\in\Z} $}.

An important problem   is to determine general and verifiable conditions on the window
function $g$, the {\it  time  sampling  }   {$ \{a_k\}_{k\in\Z}$} and the {\it  frequency  sampling   } {$\{b_n\}_{n\in\Z} $} which imply that a Gabor system is a frame.  In the regular case    many necessary and sufficient conditions  on $g$, $a$ and $b$ are known.  
 (see  e.g. \cite{DGM}, \cite{FG} and also \cite[Chapt. 11]{C1}, \cite[Chapt. 11]{He10}  	 and  the references cited in these textbooks).

Given a regular Gabor frame $\G(g,a,b)$, it is   important   to determine a {\it stability bound } $\delta >0$  so that each set   
\begin{equation}\label{def-F}\ \dsize \F=\{e^{2\pi i b (n+ \omega_{n,k  }) x}g(x-a (k+\delta_{  n,k}))\}_{n,k\in\Z}\end{equation} is a frame
whenever  $   |\omega_{n,k  } |+|\delta_ {n, k } | <\delta $.
The terms $\delta _{n,k }$ models the so-called  {\it timing jitter}, an unwelcome phenomenon of electronic systems.
The stability  of regular Gabor frames under  
perturbations of the window function and   of the time-shift  has been   investigated in \cite{F2} 
 \cite{C4}, \cite{SZ}, \cite{SZ2}.  In \cite{LW} the stability of irregular frames is investigated. 
 
 When $\delta_{n,k}=\delta_k$ and $\omega_{n,k}=\omega_n$,  
 the frame $\F$  in \eqref{def-F}  can be viewed as   an example of non-stationary Gabor frame  (NSGF).

A NSGF  is a set
$ \{e^{2\pi i \omega_n t}g_k(t)\}_{n, k\in\Z}$ where  the window function $g$  and the  frequency shifts  $\omega$  vary adaptively, depending on time and/or frequency. This adaptability makes NSGF particularly useful for analyzing non-stationary signals (e.g., signals whose spectral content changes over time, like speech or music).

  In \cite{DV1} two  authors of this paper   investigated the stability of Gabor frames $\G(\rect^{(p)},a,b)$ where $\rect(x)= \chi_{[-\frac 12, \frac 12]}(x)$ is the characteristic function of the interval $[-\frac 12, \frac 12]$ and $\rect^{(p)}(x)=\rect*... *\rect(x) $ is the $p-$times iterated convolution of    $\rect(x)$. 

In \cite[Theorem 1.1]{DV1}  it is shown that   the set   
  $\F=\{e^{2\pi i n bx}\rect( x-a(k+\delta_{n,k} ))\}_{n,k\in\Z}$ is a frame in $L^2(\R)$ provided that 
  $0<a\leq 1$ and $  4ab\sum_n\sup_k |\delta_{n,k}  |  <1$.

The   timing jitter  in the aforementioned theorem also depends on the frequencies $n$,  and  a simple example (and also Corollary \ref{C-Bala} in the next section) demonstrates  that stability cannot be achieved with timing jitter  that solely depends  on $k$ when $a=1$.

 These observations prompted further investigation into the stability of Gabor frames  $G(h; a,b )$ under frequency-dependent  timing jitter.  
 These frames   model  more realistic scenarios in communication and sampling systems, where jitter varies with frequency. Understanding how this affects frame stability ensures that such systems remain robust under these conditions. 

Furthermore,  frames in the form of 
 \eqref{def-F} 
 can be viewed   as a  nontrivial generalizations of the   NSGF, with potentially significant applications.

 \medskip
 In this paper we prove various stability results for  Gabor frames, $\G(h; a,b)$  with frequency-dependent jitter under various conditions on  the window function $h$, among which 
 the following generalization of \cite[Theorem 1.1]{DV1}.  Here, $\|\ \|_2$ denotes the standard norm in $L^2(\R)$.

\begin{theorem}\label{T-h-cmpct}
	Let    $\G(h; a,b )=\{e^{2\pi i bnx} h(x-ak)\}_{n, k\in\Z}$   be a frame with frame constants $A,\ B$. assume   that  $h\in L^2(\R)$  is supported in  the interval $(-\frac c2, \frac c2)$     with $0<a\leq  c$. 
	
	Let $\{\delta_{n,k}\}\subset (-\frac 12, \frac 12)$  be such that   
	\begin{equation}\label{cond1} \lambda =: 4\sup_k\sum_n \|h  -h(. -  a\delta_{n,k})\|_2^2  <    A.
\end{equation}
  Then, the set
	$\t \B=\{e^{2\pi i  bnx} h(x- a(k+\delta_{n,k}))\}_{n, k\in\Z}$ is a frame with frame bounds
	$$A'= \(1-\sqrt{\frac{\lambda}{  A}}\,\)A, \quad B'= \(1+\sqrt{\frac{\lambda}{  A}}\,\)B.
	$$
\end{theorem}

 Corollary \ref{C-Bala} in the next section shows that  if $h$  is not continuous in $\R$  and $a=b=c=1$,    the stability of the frame $\G(a,b,h)$ cannot be achieved with $\delta_{n,k}$   independent of $n$.

\medskip
When $h=\rect$, the  stability condition \ref{cond1}  improves that in \cite[Theorem 1.1]{DV1}.

  The condition \eqref{cond1} for the frame $\G(a,b,\rect)$ in our Theorem \ref{T-h-cmpct}is
 $$ 4 A^{-1}\sup_k\sum_n \|\rect  -\rect(. - a\delta_{n,k})\|_2^2 = 
  4a A^{-1}\sup_k\sum_n|\delta_{n,k}| < 1
$$
The stability condition   in   \cite[Theorem 1.1]{DV1}   
is
\begin{equation}\label{e-ineq-DV} 4ab \sum_n  \sup_k |\delta_{n,k}| <1.\end{equation} 
   By \cite[Lemma 2.3]{DV1},   the frame constants of $\G(\rect, a,b)$  are $A\ge \frac 1b$ and $B\leq \frac 1b[\frac 1a]$ where $[\  ]$ denotes the integer part. 
   Thus,   
   $$
   4a A^{-1}\sup_k\sum_n|\delta_{n,k}| \leq 4ab\sum_n \sup_k|\delta_{n,k}|  
   $$
   which  show that the stability condition \eqref{cond1} in our  Theorem \eqref{T-h-cmpct} improves  that in \cite[Theorem 1.1]{DV1}.    
   
   \medskip
  We also prove a stability result for   Gabor frames  $ \G(h; a,b)$ with frequency-dependent jitter when $\hat h$ has compact support  (Theorem \ref{T-hat-h-cmpct} in Section 4) and  when $h$  is  differentiable a.e. and $h'$ is in the Wiener  Amalgam space $W(L^\infty, \ell^1)$ (Theorem \ref{T-h-W} in Section 4).  We recall that a function $g\in L^p(\R)$  belongs to the {\it Wiener space}  $W(L^p, \ell^q)$,  if 
     $$
     \|g\|_{W(L^p, \ell^q)}=: \sum_m \|g(.+m)\|_{L^p(-\frac 12, \frac 12)} ^q< \infty.$$
     Here, $1\leq p\leq\infty$ and $1\leq q< \infty$.
   See e.g. \cite[Section  11.4]{He10}  for    an introduction to   Wiener amalgam spaces. 
   
Wiener amalgam spaces.  appear naturally in  Gabor frame theory. See 
    \cite{HW},  and see also  \cite [Theorem  5]{LW}, which presents  some similarities with our Theorem \ref{T-h-cmpct}. Stability conditions that depend on  the derivative of the window functions appear in \cite[Theorem 2.1]{SZ2}

       \medskip
    Our paper is organized as follows. In Section 2 we   have collected some preliminaries   and results that  are needed for our proofs. 
 In Section 3 we   prove  Theorem $1$ and some corollaries that generalize   stability  results  for B-splines proved in \cite{DV1}. 
In Section 4 we state and prove  Theorems \ref{T-h-W} and \ref{T-hat-h-cmpct}, and some corollaries.

    \section{Preliminaries}
  
  We denote with $\l \ , \ \r_2$ and $\|\ \|_2 $ the standard inner product and norm in $L^2(\R)$   
  
  Let $h\in L^2(\R)$ and $a, b>0$;
 a  Gabor  frame
 in $L^2(\R)$ is a  Gabor system   $\G=\{e^{2\pi i b_n x}g(x-a_k)\}_{n,k\in\Z}$   for which   there exists constants $ A, B>0$ such  the following  inequality holds for every $\psi \in L^2(\R)$ 
  
 \begin{equation}\label{e-frame-ineq}
 	A\|\psi\|_2^2\leq \sum_{n,k} |\l \psi,\ e^{2\pi ib_nt} h(t-a_k)\r_2|^2 \leq B|\psi\|_2^2.
 \end{equation}
  The constants $A,\ B$ are  the {\it frame bounds} of $\G $.

  When $\G=\G(h; a, b)=\{e^{2\pi i bn x}g(x-ak)\}_{n,k\in\Z}$, 
 the condition  $0< ab\leq 1$ is  necessary  for   $\G(g, a,b)$ to be a frame.
 %
 Since
 $$
 \l e^{2\pi i   b_n t}  h\left(   a_k-t\right),\ \psi\r_2 =
 \l e^{2\pi i   a_k s }  \hat h\left(   b_n-s\right),\ \hat \psi\r_2
 $$ where 
 $\hat f(s)=\int_{-\infty}^\infty f(t)e^{2\pi i st} dt$ is the Fourier transform,  we can see at once that  $\G$ is a frame if and only if 
 $\hat \G= \{e^{2\pi i a_k s} \hat h(s-b_n)\}_{n, k\in\Z}$ is a frame, and $\G$ and $\hat\G$  have the same frame constant.  We refer  the reader to \cite{He10} or to \cite{C1}  for  an introduction to Gabor frame

\subsection{Non-stationary  Gabor frames}
To the best of our knowledge, non-stationary  Gabor frames  (NSGF)  $\F=\{e^{2\pi i \omega_n x} h_k (x)\}_{k,\, n \in\Z} $   were first introduced by Jaillet,  Balazs, D\"orfler and Engelputzeder in 
\cite{Ba}  (see also  an early example in  \cite{Ru}).  Other references on the NSGFs  are \cite{B}, \cite{Do}\cite{Li}.

When the $h_n$ have compact support, the following result proved by P. Balazs \& al. in   \cite{B}
 generalizes  the classical   painless non-orthogonal expansions for regular Gabor frames:
\begin{proposition}\label{P-Balaz}
  Let $\F=\{e^{2\pi i \omega_n x} h_k (x)\}_{k,\, n \in\Z} $ be a frame;   if supp$(h_k)\subset [c_k,\, d_k]$  for every $k\in\Z$, and   $b_k$ is chosen so that $d_k-c_k\leq \frac{1}{b_k}$,  the frame operator $S(f)=\sum_{n,  k}\l f, g_{k,n}\r g_{k,n} $ of  $\F$  is given by 
$$
Sf (t)= f(t)\sum_k \frac{1}{b_k} |h_k(t)|^2. 
$$ 
\end{proposition}
This proposition shows that  the set $\F$ is a frame if and only if   there exist constant $A,\, B>0$ for which 
$$
A\leq \sum_k \frac{1}{b_k} |h _k(t)|^2\leq B.
$$
 If the    supports of the $h_k$ do not  cover $\R$, then $\F$ is not a frame.
So if $h_k= h(t-a(k+\delta_k))$  with  supp $(h)= [-\frac c2, \frac c2]$   and   $0<a\leq 1$,    the support of $h_k$ intersects that of $h_{k+1}$ if 
$$
\frac c2 +a(k+  \delta_k) \ge -\frac c2  +a(k+1+ \delta_{k+1}).
$$
 
 Thus,  if
$ a(\delta_{k+1}-\delta_k)\leq c-a$ for every $k\in\Z$,  the  support of the $f_k$ overlap.

This observations prove the following  corollary of Proposition \ref{P-Balaz}
\begin{corollary}\label{C-Bala}
Let $\G=\{e^{2\pi i n\omega_n x} h(x-a k)\}_{k,\, n \in\Z} $ be a frame. Assume $g\in L^2(\R)$  with support in $[-\frac c2, \frac c2]$ and  that $0<a\leq c$.
 
  If    $\{\delta_k\}_k\subset \R$ is such that $   \delta_{k+1}-\delta_k  \leq \frac{c-a}{a}$  for every $k\in\Z$, the set $ \{e^{2\pi i n\omega_n x} f(x-a(k+\delta_k))\}_{k,\, n \in\Z} $  is a frame.

\end{corollary}

 \medskip

\subsection{Paley-Wiener and Kadec theorems}
%
One of the fundamental stability criteria for bases in Banach spaces, and historically the first, is due to R. Paley and N. Wiener (see \cite {PW}).   The  following generalized   Paley-Wiener theorem is in     \cite{C2}.  

\begin{theorem}[Paley--Wiener theorem for frames]\label{T-PaleyWiener}
	Let $\{x_n\}_{n}$  be a frame  in a Hilbert space $H$ with frame bounds  $A$ and $B$.   Let $\{y_n\}_{n}\subset H$; suppose that  there exists 
	  $\lambda, \mu \ge 0$, with $\lambda+\frac{\mu }{\sqrt A} <1$,  for which   
	  \begin{equation}\label{e-PW}	  \mbox{$\dsize \lambda+\frac \mu{\sqrt A} <1$\quad  and \quad
	  $\dsize \norm{\sum_n a_n(x_n-y_n)}\leq \lambda \|\sum_n a_n x_n\|
	  	+\mu\(\sum_n |a_n|^2\)^{\frac 12}  $}
  \end{equation}
 for every   finite set of scalars $\{a_n\}\subset \C $. Then,
    $\{y_n\}_n $ is a frame   with bounds 
	$
	\(1-(\lambda+\frac \mu{\sqrt A} )\) A $ and $    \(1+\lambda+\frac \mu{\sqrt A}) \) B $
\end{theorem}
 If $\mu=0$  and $\{x_n\}_{n}$  is a Riesz basis, then  Theorem \ref{T-PaleyWiener} reduces to the classical  Paley-Wiener theorem   and   proves that also $ \{y_n\}_{n}$  is a  basis.

\medskip
A celebrated stability result for  the standard orthogonal basis $\{ e^{ i n x}\}_{n\in\Z}$ on $L^2(-\pi,\pi)$ is the classical Kadec-$\frac 1 4$ theorem.  
It states that the set  $\{ e^{ i  (n+\delta_n) x}\}_{n\in\Z}$ is a basis on  $L^2(-\pi,\pi)$ whenever $\sup_n|\delta_n|<\frac 14$. An example by  Levinson shows that the constant $\frac 14$  cannot be replaced by any larger constant. For the proof of Kadec's theorem 
 see \cite{K} and also\cite[Chaper 1, Theorem 14, page 42]{Y}.   
 
  The proof of   Kadec's theorem   in \cite{Y}  uses  the fact that 
  the sequence
  $$S (t)=  A_0 (\delta)  + \sum_{m=1}^\infty  A_m( \delta) \cos (    m t ) + \sum_{m=1}^N B_m( \delta) \sin (    t (m-\mbox{$\frac 12$}) ) ,$$  where 
  $
  	A_0 (\delta)   = 1-   \frac{  \sin (    \pi  \delta)}{  \pi   \delta}$, 
  $ A_m(\delta) =      \frac{ 2(-1)^{m}   \delta   \sin (   \pi  \delta )}{\pi  \left(m^2-   \delta^2 \right)} $,  $B_m(\delta) =     i  \frac{ 2(-1)^{ m}     \delta    \cos (    \pi   \delta)}{\pi  \left((m-\frac 12)^2-    \delta^2  \right)} 
$
  converges pointwise to $f(t)=1-e^{i\delta t}$ in the interval $[-\pi,\,\pi]$. This important fact is used to  show that   for every  sequence $\{\delta_n\}_n\subset (-\frac 14, \frac 14)$ and   every  finite  set $\{a_k\}\subset \C$,  the following inequality holds with $\delta=\sup_n|\delta_n|$.
  \begin{equation}\label{e-Kadec}\|\sum_n a_n \(e^{ i n t}  -e^{  i (n+\delta_n)  t}\)\|_{L^2( -\pi
   	\,  \pi)} \leq  (1-\cos(\pi\delta  )+ \sin(\pi\delta  )  ) \(\sum_n |a_n|^2\)^{\frac 12}. \end{equation}
   Since  $1-\cos(\pi\delta  )+ \sin(\pi\delta  ) <1$ when $0\leq \delta<\frac 14$,  we can apply  Theorem \ref{T-PaleyWiener}  with $\lambda=0$,  $ \mu=1-\cos(\pi\delta  )+ \sin(\pi\delta  ) $ and $A=1$ and conclude that the set $\{e^{i(n+\delta_n) t}\}_{n\in\Z}$ is a  Riesz basis of $L^2(-\pi, \pi)$.
  
  \medskip
 We prove the following

\begin{lemma}\label{L-Gen-K}
  Let $M>0$  be an integer;   let    $\{\delta_n\}_n\subset\R$, with $\delta=:\sup_n|\delta_n| <\frac 1{4M}.$   Then, for every  finite sequence $\{a_n\}\subset \C$, we have that 
 	\begin{equation}\label{e-Kadec2}\|\sum_n a_n e^ { 2\pi i n t} (1-e^{ 2\pi i \delta_n  t})\|_{L^2( -\frac M2
 		\,  \frac M2)}\leq \sqrt{\frac M {2\pi}}\(1-\cos( \pi\delta M)+ \sin(  \pi\delta M)  \) \sum_n |a_n|^2. \end{equation}
\end{lemma}

\begin{proof}     
 Letting  $t=   \frac { M}  { 2\pi} s$  in the $L^2$ norm in the inequality \eqref{e-Kadec2}, 
 we obtain
	$$ S=:\|\sum_n a_n e^ { 2\pi i n t} (1-e^{ 2\pi i \delta_n  t})\|_{L^2( -\frac M 2
		\, \frac M 2)}= \sqrt{\frac M {2\pi}}\|\sum_n a_n e^ {  i M n s} (1-e^{   i M\delta_n  s})\|_{L^2( -\pi
		\,  \pi)}.
	$$
We let $b_m= \begin{cases} 0 & \mbox{if $m\ne  Mn$}
		\cr a_n & \mbox{if $m=  Mn$} \cr\end{cases}$ and  $\delta_m= \delta_n $ whenever  $m\in [ nM,\,  (n+1)M)$.    In view of  \eqref{e-Kadec}, we obtain
	$$
	S= \sqrt{\frac M {2\pi}}\|\sum_m b_m e^ { i m s} (1-e^{   i M\delta_{m}  s})\|_{L^2( -\pi
		\,  \pi)} \leq \sqrt{\frac M {2\pi}} (1-\cos( \pi M\delta  )+ \sin( \pi M\delta  )  ) \(\sum_n |b_n|^2\)^{\frac 12}
	$$
	$$
	= \sqrt{\frac M {2\pi}} (1-\cos( \pi M\delta  )+ \sin( \pi M\delta  )  ) \(\sum_n |a_n|^2\)^{\frac 12}
	$$ 
	as required.
	 \end{proof}

\medskip{\it Remark.}
We recall the elementary inequalities
$ \dsize
 \sin t\leq t$ and $   0 \leq 1-\cos t  \leq \frac{t^2}{2}. 
$ 
When $0\leq t \leq \frac{\pi }2$, we also have  $\frac{2t }\pi\leq \sin t$.  Thus, when  $ M\delta <\frac 14$ we have
\begin{equation}\label{e- element}
2  \delta M  \leq  (1- \cos( \pi \delta M) +\sin( \pi \delta M))\leq  \pi \delta M  + \frac{( \pi \delta M)^2}{2}
\end{equation}
\subsection{Poisson summation formula}

Let $f  \in C(\R)\cap L^2(\R)$. The  basic Poisson summation  formula states that 
$$\sum_n   f(x+n)=\sum_k\hat f(k) e^{2\pi i k x}, \quad x\in \R.
$$
Given the parameters $  P>0$, we can let  
$ s_Pf(x)=\sum_n   f(x+Pn)=\sum_n   f_P(x/P+ n)  
$ 
where $f_P(t)=f(Pt)$. Recalling that 
$\widehat f_P(y)=\frac 1P\hat f( y/P) $ and 
using   the basic Poisson formula above,  we obtain
$$
s_Pf(x)= \sum_n   f(x+Pn)= \sum_n   f_P(x/P+n)=\frac 1P\sum_k \hat f(k/P) e^{2\pi i kx/P}
$$
Letting $f=\hat F$ and recalling that $\hat{\hat F} (t)=F(-t)$, we obtain 
\begin{equation}\label{e-Poiss-2}
 \sum_n   \hat F(y+Pn)= \frac 1P\sum_k F(-k/P) e^{2\pi i ky/P}= \frac 1P\sum_k F( k/P) e^{-2\pi i ky/P}.
\end{equation}
We prove the following
\begin{lemma}\label{L-Poisson}
a)  Let $g\in   L^2(\R)$.  Then, for every $P>0$ and a.e. $y\in\R$,
\begin{equation}\label{1.2}
 \sum_n  | \hat g(y+Pn)|^2=\frac 1P\sum_{k}  e^{-2\pi i ky/P}\int_{-\infty}^\infty g(s)\overline g(s-k/P)dt
\end{equation}
 where $\overline g$  denotes the complex conjugate of $g$.
 
 b) If  $g$ is supported in $(t-\frac 12, t+\frac 12 )$  for some $t\in \R$ and if $P\leq 1$, we have that  
\begin{equation}\label{1.1}
 \sum_n  | \hat g(y+ Pn)|^2=   \frac 1P \|g\|_2^2.
\end{equation}
\end{lemma}
\begin{proof}
 Let us show first that for a given $\psi\in L^2(\R)$  and for a.e. $y\in\R$, we  have that
$|\hat \psi(y)|^2=\widehat{\psi*\psi_-}(y)$,
where   $\psi_-(x)=\overline \psi(-x)$.
Indeed,
$$|\hat\psi (y)|^2=\hat\psi(y)\overline{\hat\psi(y)}=
\hat\psi(y)\int_{-\infty}^\infty\overline\psi(t)e^{2\pi i y t}dt,
$$
and with the change of variable $t\to -t$ in the integral, we obtain 
$$|\hat\psi (y)|^2= 
\hat\psi(y)\int_{-\infty}^\infty\overline\psi(-t)e^{-2\pi i y t}dt= \hat\psi(y)\hat\psi_-(y)=\widehat{\psi*\psi_-}(y).
$$
Using this identity and \eqref{e-Poiss-2},  we obtain 
$$\sum_n  | \hat g(y+Pn)|^2=  \sum_n \widehat{g*g_-}(y+Pn)
=
     \frac 1P\sum_k g*g_-(k/P) e^{-2\pi i ky/P}$$$$= \frac 1P\sum_ke^{-2\pi i ky/P}\int_{-\infty}^\infty g(s)\overline g(s-k/P)ds
 $$
which is \eqref{1.2}.

 If $g$ is supported in $( t-\frac 12,\   t+\frac 12)$ and  $P\leq 1$,  the  supports of $g(s)$  and   $g(s-k/P)$ are disjoint whenever  $k\ne 0$.  Thus,  the sum above reduces to 
  $ 
  \frac 1P\int_{-\infty}^\infty |g(s)|^2 ds=\frac 1p\|g\|^2_2
  $ 
   and \eqref{1.1} is proved.
   
\end{proof}

\subsection{The short-term Fourier transform}

Let  $g \in L^2(\mathbb{R})  $ with $\|g\|_2\ne 0$.   The {\it short-time Fourier transform}   with respect to $g$, also called {\it Gabor transform},  is  the operator  $ \mathcal{F}^{g} :L^2(\R  )\to L^2(\R\times \R)$  
	 \begin{equation}\label{e-STFT}
	(\mathcal{F}^{g} f)(\xi,t)=\int_{-\infty}^{\infty} e^{-  2 \pi i  \xi s}\, f(s)\, g(s-t)\, d s, \quad t,\xi\in\mathbb R
\end{equation}

For a given $t\in\R$, we can let  $g_t(x)=g(x-t)$ and observe  that $\mathcal{F}^{g} f $  is the Fourier transform of $fg_t$.  
 This windowed Fourier transform provide a tool for time-frequency localization, and shares many  properties with the Fourier transform. For example,    it is proved  \cite[Section 11.1 ]{C1} that   for any
 $f_1, f_2, g_1, g_2 \in L^2(\mathbb{R})$, we have 
 	\begin{equation}
 		\label{eq:generalizedpancherelstft}
 		\int_{-\infty}^{\infty} \int_{-\infty}^{\infty} (\mathcal{F}^{g_1} f_1)(\xi,t) \overline{(\mathcal{F}^{g_2} f_2)(\xi,t)} d t d \xi=\left\langle f_1, f_2\right\rangle_2\left\langle g_2, g_1\right\rangle_2.
 	\end{equation}
 From this formula we obtain a generalized  Plancherel identity
 \begin{equation}
 	\label{eq:plancherelstft}
 	\|\mathcal{F}^{g} f\|^2_{L^2(\R^2)}=\|f\|^2_2\, \|g\|^2_2.
 \end{equation}

Other  properties of the short-term Fourier transform are   described  e.g in  the aforementioned  \cite[Section 11.1 ]{C1} and in \cite{Gr}.  

We  will need  the following
\begin{proposition}\label{P-Pois}
 Let $g\in L^2(\R)$ with  support in   $ (-\frac 12, \frac 12)$. For every  $f\in L^2(\R)$,   every $t\in\R$ and every  $P\leq 1$, we have that 
	$$\sum_n |\F^g f(Pn, t)|^2 =  \frac 1P\int_{-\infty}^{\infty}  |f(s)|^2 |g(t-s)|^2  ds $$
	\end{proposition}
 
\begin{proof}
	
	Fix $t\in\R$.     By Lemma \ref{L-Poisson}
	 with  $y=0$,
	 $$
	\sum_n  | \F^g(f(Pn, t)|^2= 	 \frac 1P\sum_n |\widehat{g_tf}(Pn)|^2 =
	\int_{-\infty}^\infty |fg_t|^2(s) ds=  \frac 1P\int_{-\infty}^{\infty}  |f(s)|^2 |g(t-s)|^2  ds 
	$$
	as required.
\end{proof}

\section{Proof of theorem \ref{T-h-cmpct} and corollaries}

In this section we prove Theorem \ref{T-h-cmpct}, but first we  recall some basic facts about  the stability of Gabor frames    and about Gabor frames  with compactly supported window. 

\medskip
Given  the frame   $\G(h; a,b)=\{e^{2\pi i bnx} h(x-ak)\}_{n, k\in\Z}$, we   find sufficient conditions that ensure that the set $$  \tilde \G=\{e^{2\pi i bnx} h(x-a(k+\delta_{n,k}))\}_{n, k\in\Z}$$ is a frame. 
As usual,    the frame constants of $\B$  will be denoted  with $A$ and $B$. By   Theorem \ref{T-PaleyWiener},    if we  show that, for every  finite set $\{a_{n,k}\}\subset \C$, we have that
\begin{equation}
	 \label{eq:norm}
	\left\|\sum_{n,k}\, a_{n,k} \, e^{2\pi i b n x }\left(h\left(x - a k  \right)-h\left(x-a(k+\delta_{n,k} )  \right)\right)\right\|^2_2 \leq \mu A \sum_{n,k}\, |a_{n,k}|^2
\end{equation}
for some $0<\mu <1$, we have proved that  $\t \G$ is a frame with frame bounds
 $ A'= (1-\frac{\mu}{\sqrt A})A, \quad B'= (1+\frac{\mu}{\sqrt A})B.
 $ 
 
A standard application of  Plancherel theorem and of the properties of the Fourier transform  show that \eqref{eq:norm} is equivalent to 
 \begin{equation}
 	\label{eq:norm2}
 \big\|\sum_{n,k}\, a_{n,k} \( e^{2\pi i a k( x-bn)   } - e^{2\pi i a ( k+ \delta_{k,n})(x-bn)}\)\hat h(x-bn) \big\|_2^2 \leq \mu A \sum_{n,k}\, |a_{n,k}|^2.
\end{equation}
  Both inequalities \eqref{eq:norm}  and \eqref{eq:norm2} will be used in our proofs.
  
 \medskip 
   
   \medskip

In the seminal paper  \cite{DGM} it is proved  that when   $h$ is supported in an interval $ [-\frac c2,\, \frac c2]$,  with $c>0$,     then  $\G (h; a,b)$  is a frame  if and only  if $a\leq c \leq b^{-1}$        and the inequality 
\begin{equation}\label{e-NES condition} Ab\leq \sum_{k\in\Z}|h(x-ak)|^2 \leq Bb\end{equation}
 holds for a.e. $x\in\R$ with constants $A,\ B>0$ 
See also \cite[Theorem 11.6]{He10}.
 The optimal    constants $A,\ B$ in the inequality \eqref{e-NES condition} are the  optimal frame constants of $\G(h; a,b)$.
It is proved in \cite[Thm 11.8]{He10} that 
  $$A  = \frac 1a\inf_x\sum_k| \widehat h  (x-bk)|^{2 }, \quad   B = \frac 1a\sup_x\sum_k|\widehat h   (x-bk)|^{2 }. 
  $$
 In view of the definition of frames \eqref{e-frame-ineq},   a change of variables in the inner product yields
  \begin{equation}\label{e-frame1}
  	A\|\psi\|_2^2\leq a^2\sum_{n,k} |\l \psi_a, e^{2\pi i banx} h_a(x- k)\r_2|^2 \leq B \|\psi_a\|_2^2
  \end{equation}
  where we have let $f_a(t)= f(at)$. Since $a^{-1}\|\psi \|_2^2= \|\psi_a\|_2^2$,  the inequality \eqref{e-frame1} can be written as
  
  $$
 a^{-1} A\|\psi_a\|_2^2\leq  \sum_{n,k} |\l \psi_a, e^{2\pi i banx} h_a(x- k)\r_2|^2 \leq a^{-1} B \|\psi_a\|_2^2,
  $$
 and we can  conclude that   $\G(h; a,b)$ is a frame  with frame constants  $A$ and $B$ if and only if  the set 
  $$\G(h_a; 1,ab   )= \{e^{2\pi i abnx} h_a(x- k)\}_{n, k\in\Z} $$ is a frame  with frame constants   $a^{-1} A$ and $a^{-1}B$. Note that the support of $h_a$ is in the interval    $[-\frac{c}{2a }, \frac{c}{2a }] \supset  [-\frac{1}{2 }, \frac{1}{2} ] $.

   \subsection{Proof of Theorem \ref{T-h-cmpct}}
 
 By the previous observations, 
  we can reduce matters to  proving that if the assumptions of the theorem are satisfied,   the set $$\t \B=:\{e^{2\pi i abnx} h_a(x- (k+\delta_{n,k}))\}_{n, k\in\Z}$$   is a frame whenever the set $ \B=\G(h_a; 1,ab   ) $  is a frame.

For simplicity we  assume that $ h_a $ is supported  in $[-\frac 12, \frac 12]$, but the proof is similar also in the other case.  
To simplify  the notation in the proof,  
we will replace $ k+\delta_{n,k}$ with $\mu_{n,k}$    when there is no risk of confusion.  


By  Theorem \ref{T-PaleyWiener}, the set $\t\B$ is a frame if we prove that    for every  finite sequence $\{a_{n,k}\}\subset \C$, the  following inequality holds with constant $\lambda<1$.
\begin{equation}\label{def-I}
\left\|\sum_{n,k}\, a_{n,k} \, e^{2\pi i   ab n x }\left(h_a \left(  x-\mu_{n,k} \right)-h_a \left(  x-k \right)\right)\right\|^2_2 \leq \lambda  a^{-1}A \sum_{n,k}\, |a_{n,k}|^2. \end{equation}

Fix $k\in\Z$ and let   $$f_k(x)=\sum_{ n}e^{2\pi i   ab n x } a_{n,k} \, \left(h_a \left(   x-\mu_{n,k} \right)-h_a \left(  x-k \right)\right) .$$
The inequality \eqref{def-I}  follows if we prove that for every   finite sequence $\{a_{n,k}\}\subset \C$, with 
$ \sum_{n,k}\, |a_{n,k}|^2 =1$, the   inequality  
\begin{equation}\label{def-I2}\|\sum_k f_k\|_2^2  \leq    \lambda a^{-1} A   
\end{equation}
holds with constant $\lambda<1$.
We let $d_j=\sup_n|\delta_{n,j}|  $, and 
$$ I_{ j}'=[ j-d_{ j}-\frac 12, \  j+d_{ j}- \frac 12), \quad I_{j } =[j +d_{ j} -\frac 12, \ j-d_{ j}+\frac 12).
$$
Clearly,  $\R=\bigcup_j I_j  \cup I_j'$, and since   $|d_k|<\frac 12$, all intervals are disjoint.

	We show that each $f_k$ is  supported in  the  interval  $[k-\frac 12-d_k, \ k+\frac 12 + d_k) = I_{ k}'\cup I_k \cup I_{ k+1}'$.
	
Indeed, for every  $n\in\Z$,   the  support of the function   $x\to  h_a \left(   x-k-\delta_{n,k} \right)-h_a \left(   x -k\right)$ is in the interval  $  (k-\frac 12, \, k+\delta_{n,k}+\frac 12)\subset (k-\frac 12-d_k,\ k+ d_{ k}+\frac 12)$    whenever $ \delta_{n,k}\ge 0$,  and   in the interval 
 $  (k+\delta_{n,k} -\frac 12, \, k+  \frac 12)\subset (k-d_{ k}-\frac 12,\  k+\frac 12+d_k) $ whenever $  \delta_{n,k} \leq   0$. 
Note that   only $f_k$ is supported in $I_k $ but  both $f_k$ and $f_{k-1}$ can be supported  in $ I_k'$.

In view of the elementary inequality $(a+b)^2\leq 2(a^2+b^2)$, we can write 
\begin{equation}\label{e- tot-sum}
\|\sum_k f_k\|_2^2 =\|\sum_k \chi_{I_k '}f_k + \sum_k \chi_{I_k   }f_k\|_2 ^2  \leq  2\| \sum_k \chi_{I_k   }f_k\|_2^2+2\|\sum_k \chi_{I_k'}f_k \|_2^2. 
\end{equation}

Let us estimate   the  first sum in \eqref{e- tot-sum}.  Since $I_k \cup I_j=0 $ whenever $k\ne j$ and only $f_k$ is supported in $I_k$,   we can write
$$  \sum_k  \|\chi_{I_k   }f_k\|_2^2 =  
\sum_k\int_{k-\frac 12+d_k}^{k+\frac 12-d_k}\big|\sum_n e^{2\pi i   ab n x }  a_{n,k}  (h_a(x-k)- h_a(x-\mu_{n,k}))\big|^2 dx. $$
After applying the the Cauchy-Schwartz inequality      to the sum in $n$ and a change of variable in the integrals, we obtain
$$
\|\sum_k \chi_{I_k } f_k\|_2^2\leq \sum_k  \sum_n | a_{n,k} |^2   \sum_n\int_{k-\frac 12+d_k}^{k+\frac 12-d_k}|    (h_a (x -k)- h_a(x -k-\delta_{n,k})) |^2dx 
 $$\begin{equation}\label{Sum1}= \sum_k  \sum_n | a_{n,k} |^2   \sum_n\int_{ -\frac 12+d_k}^{ \frac 12-d_k}|     h_a (x  )- h_a(x  -\delta_{n,k}))  |^2dx.  \end{equation}

We now  estimate the  second sum in \eqref{e- tot-sum}.
We have observed that  the $I'_k$ are disjoint and  both $f_j$    and  $f_{j-1}$ can be supported in $I'_j$. Using again the elementary inequality $(a+b)^2\leq 2(a^2+b^2)$, we can write
\begin{equation}\label{e-3sum}\|\sum_k\chi_{I_{ k} '} f_k\|_2^2= \sum_k \|\chi_{I_{ k} '} (f_k+f_{k-1})\|_2^2 \leq 2\sum_k \| \chi_{I_{ k} '} f_k\|_2^2 +
2\sum_k \| \chi_{I_{ k} '} f_{k-1}\|_2^2. 
\end{equation}
Recall that $I_{ j}'=[ j-d_{ j}-\frac 12, \  j+d_{ j}- \frac 12)$.  Arguing as in the proof of \eqref{Sum1}, we obtain
$$
\sum_k \|  \chi_{I_{ k}'} f_k\|_{2}^2 \leq \sum_k  \sum_n | a_{n,k} |^2   \sum_n\int_{ -\frac 12-d_k}^{ -\frac 12+d_k}|     h_a (x  )- h_a(x  -\delta_{n,k})  |^2dx,
$$ 
and  $$
\sum_k \|  \chi_{I_{ k}'} f_{k-1}\|_{2}^2 \leq \sum_k  \sum_n | a_{n,k} |^2   \sum_n\int_{  \frac 12-d_k}^{  \frac 12+d_k}|     h_a (x  )- h_a(x  -\delta_{n,k})  |^2dx.
$$ 
  From these inequalities  and \eqref{e-3sum}, we obtain 
$$ \|\sum_k\chi_{I_{ k} '} f_k\|_2^2 \leq 2 \sum_k  \sum_n | a_{n,k} |^2   \sum_n\(\int_{  \frac 12-d_k}^{  \frac 12+d_k} \!\!\!\! |     h_a (x  )- h_a(x  -\delta_{n,k})  |^2dx + \int_{ - \frac  12-d_k}^{  -\frac 12+d_k}  \!\!\!\!   | h_a (x  )- h_a(x  -\delta_{n,k})  |^2dx\).
$$
This inequality  and  \eqref{Sum1}, combined with    \eqref{e- tot-sum},   yield
\begin{equation}\label{e-last}
\sum_k \| f_k\|_{2}^2 \leq 4\sum_k  \sum_n | a_{n,k} |^2   \sum_n\int_{  -\frac 12-d_k}^{  \frac 12+d_k}|     h_a (x  )- h_a(x  -\delta_{n,k})  |^2dx.
\end{equation}
Recall  that  the support of the function $x\to h_a (x  )- h_a(x  -\delta_{n,k}) $ is in the interval $( -\frac 12-d_k, \  \frac 12+d_k)$; letting $m(t)=:\|   h   - h (.-t) \|_2$, we have that
 $$
 \int_{  -\frac 12-d_k}^{  \frac 12+d_k}|     h_a (x  )- h_a(x  -\delta_{n,k})  |^2dx= 
 \int_{  -\infty}^{  \infty}|     h  (ax  )- h (ax  -a\delta_{n,k})  |^2dx
 $$
 $$
 = 
 a^{-1} \|   h  - h (\cdot  -a\delta_{n,k})  \|^2_2=  a^{-1} m(a\delta_{n,k})^2.
 $$
 We have used the change of variables $ax\to x$ in the integral.  We can write the inequality \eqref{e-last} as
 $$\sum_k \| f_k\|_{2}^2 \leq 4 a^{-1}\sum_k  \(\sum_n | a_{n,k} |^2   \sum_nm(a\delta_{n,k})^2\)
 $$
from which follows that $$\sum_k \| f_k\|_{2}^2 \leq 4 a^{-1} \sup_k \sum_nm(a\delta_{n,k})^2 \sum_{n,k} | a_{n,k} |^2  = 4 a^{-1}\sup_k \sum_nm(a\delta_{n,k})^2.
$$
 
The inequality \eqref{def-I2} is satisfied if 
$$
 \lambda=: 4\sup_k \sum_nm(a\delta_{n,k})^2  <  \sqrt A.
$$  
  which is \eqref{cond1}. The theorem is proved. $\Box$

\subsection{ Corollaries}
Let   $h\in L^2(\R)$ with support in the interval $\left(-\frac{c}{2}, \frac{c}{2}\right)$.   As observed at the beginning of this section,   if   $\G(h; a,b )$  is a frame,   necessarily  $a \leq c \leq \frac 1b$;  we have also observed that if  $\G(h; a,b )$ is a frame with frame constants $A$ and $B$,  then 
 $ \G(h_a; 1, ab)$ is a frame with frame constants $a^{-1}A$ and $a^{-1}B$.  Here and throughout this subsection, $f_a(x)=f(ax)$. 
 
 \medskip
 For a given $p\in\N$, we consider the $p$-times  iterated convolution $h^{(p)}(x) =:\underbrace{ h * \dots * h}_{p\ times}(x)$.

 Using the properties of the convolution,   it is easy to verify that   the  $h^{(p)}(x)$ are supported in the interval $\left(-\frac{cp}{2}, \frac{cp}{2}\right)$ and  are continuous for $p\ge 2$. 
   When $h=\rect$, the functions  $\operatorname{rect}^{(p)}(x)  $  are   piecewise polynomial of degree $p-1$ and a prime example of a $B$-spline of order $p-1$ (see e.g  \cite{PB}, \cite{Db}  and other classical references on  B-splines).
   
   \medskip
We  consider the set 
 $\G(h^{(p)} ; pa, b/p )$.    By the previous consideration, this set is a frame  with constants $A_p$ and $B_p$ if and only if   the set
 $\G(h^{(p)}_{ap} ; 1, ab )$ is a frame  with constants $(ap)^{-1} A_p$  and  $(ap)^{-1} B_p$.
 Note that the assumption  $a \leq  c \leq \frac 1b$ yields  the necessary condition     
 $ pa  \leq pc \leq   \frac{1}{b/p}$   that $\G(h^{(p)} ; pa, b/p )$ needs to satisfy in order to be a frame.

 \medskip

 %
  
The following  easy lemma is perhaps already known, but we prove it here for the convenience of the reader.
\begin{lemma}\label{L-Ap} Assume that  $\G(h; a,b)$ is a frame   with frame constants $A,\ B$,  with $h$    supported in  $(-\frac c2, \frac c2)$. Assume also that 
	$$N(h)=:\sup_x\sum_k |\hat h(x+kb)|^{\frac 43}  <\infty.
	$$
 Then for every $p\in\N$, the set 	
	 $\mathcal{G}(h^{(p)}; pa,\, b/p)$   is a frame with frame  constants $ A_p, \ B_p $ that satisfy the relations
	 $$B_p = a^{p-1} B^p,$$   
	  \begin{equation}\label{e-Ap-1}A_1=A; \quad  A_{p+1}  = a^{\frac 1p+1 }\( \frac{  A_{p}  ^{\frac 1p}}{N(h) } \)^{2(p+1)}
	\end{equation}
\end{lemma}
\begin{proof} By \cite[Theorems  11.4, 11.8]{He10}   and the previous considerations,  the set  $\mathcal{G}(h^{(p)};  pa, b/p)$   is a frame if  
	$$A_p =: \frac 1{ap}\inf_x\sum_k| \widehat{h^{(p)}} (x-bk/p )|^{2 }>0 \quad   B_p=: \frac 1{ap}\sup_x\sum_k|\widehat {h^{(p)}} (x-bk/p)|^{2 } <\infty.
	$$
Recall that  $A_1 =A $ and $B_1 =B $ are the frame constants of $\G(h; a,b)$.

 The Fourier transform of the $ p $-fold convolution of $ h(x) $ is   the Fourier transform of $ h(x) $ raised to the power $ p $, and so 
\begin{equation}\label{e-Ap}A_p = \frac 1{ap}\inf_x\sum_k| \hat h  (x-bk/p)|^{2p }, \quad   B_p= \frac 1{ap}\sup_x\sum_k|\hat h  (x-bk/p)|^{2p }.\end{equation}
Every $k\in\Z$ can be written as $k=k'p+j$ with $0\leq j\leq p-1$.  Thus,
$$A_p =
\frac 1{ap}\inf_x\sum_{j=0}^{p-1}\sum_{k'} | \hat h  (x-bk' -jb/p)|^{2p } \ge \frac 1{ap}\sum_{j=0}^{p-1} \inf_x\sum_{k'} | \hat h  (x-bk'-bj/p)|^{2p }=\frac 1{a }\inf_x \sum_{k'} | \hat h  (x-bk')  |^{2p }.
$$$$\quad   B_p= \frac 1{ap}\sup_x\sum_{j=0}^{p-1} \sum_{k'}| \hat h  (x-bk/p-bj/p)|^{2p } \leq \frac 1{ap}\sum_{j=0}^{p-1} \sup_x\sum_{k'} | \hat h  (x-bk'-bj/p)|^{2p } = \frac 1a \sup_x\sum_{k'} | \hat h  (x-bk' )|^{2p }.
  $$
 We   let $B'_p=:\sup_x\sum_{k'} | \hat h  (x-bk' )|^{2p } $ and 
 $A'_p=:\inf_x\sum_{k'} | \hat h  (x-bk' )|^{2p }$.    Thus, $$A_p\ge \frac 1a A'_p \mbox{\quad and \quad}  B_p\leq  \frac 1a B'_p.$$

and  it is enough to estimate  $A'_p$ and $B'_p$.
 
Recall   that for every $1\leq r\leq q<\infty$ and every sequence $\{x_n\}_n\subset \C$,  the inequality 
  \begin{equation}\label{ineq-1}
 \(\sum_k|x_k| ^q\)^{\frac 1q}\leq \(\sum_k|x_k| ^r\)^{\frac 1r}
  	\end{equation}
holds.    From this inequality follows that  $(\sum_k|x_k|)^r\ge \sum_k|x_k|^r$ 
whenever  $r\ge  1$. Thus,    
	$$
   \sum_{k'} | \hat h  (x-bk' )|^{2p } \leq  \sup_x\(\sum_k|\hat h  (x-bk)|^{2 }\)^p=   \sup_x\( \sum_k|\hat h  (x-bk)|^{2 }\)^p=  (a B)^p,
	$$ 
and  so  $B'_p\leq   (aB)^p$.

To estimate $A'_p$, we use  H\"older's inequality  and induction on $p$.   Letting $$N(h)=: \(\sum_k|\hat h  (x-bk)|^{ \frac 43 } \)^{\frac 34},$$ we start with the inequality 
	$$
 \sum_k|\hat h  (x-bk)|^{2 } \leq \(\sum_k|\hat h  (x-bk)|^{ 4 } \)^{\frac 1{4}}\(\sum_k|\hat h  (x-bk)|^{ \frac 43 } \)^{\frac 34} 
=  \(\sum_k|\hat h  (x-bk)|^{ 4 } \)^{\frac 1{4}}N(h) $$
 from which we obtain
  	$$
	 \sum_k|\hat h  (x-bk)|^{4 } \ge   \(\frac{	\sum_k|\hat h  (x-bk)|^{2  }}{N(h)} \)^{4} \ge   \(\frac{aA}{ N(h)}\) ^{ 4}.
	$$
	and    $(A_2')^{\frac 12}\ge  \(\frac{A_1'}{N(h)}\)^2=\(\frac{aA}{N(h)}\)^2$.

	When $p\ge 2$, we can write 
	$$
	\(\sum_k|\hat h  (x-bk)|^{2p }\)^{\frac 1p} \leq \(\sum_k|\hat h  (x-bk)|^{ 2(p+1) } \)^{\frac 1{2(p+1)}}\(\sum_k|\hat h  (x-bk)|^{ q} \)^{\frac 1q} 
	$$
	where $q$ satisfies  $\frac 1p= \frac{1}{2(p-1)}+\frac 1q$.
	Since $p\ge 2$, we have that $q\ge 2>\frac 43$. By the inequality  \eqref{ineq-1},    $\(\sum_k|\hat h  (x-bk)|^{ q} \)^{\frac 1q} \leq N(h)$; thus,
	$$
	\(\sum_k|\hat h  (x-bk)|^{2p }\)^{\frac 1p} \leq \(\sum_k|\hat h  (x-bk)|^{ 2(p+1) } \)^{\frac 1{2(p+1)}}N(h) 
	$$
from which follows that 
	$$
	 \(\sum_k|\hat h  (x-bk)|^{ 2(p+1) } \)^{\frac 1{2(p+1)}} \ge   \frac{	\(\sum_k|\hat h  (x-bk)|^{2p }\)^{\frac 1p}}{N(h)   }   \ge  \frac{(A_p')^{\frac 1p}} {N(h) } 
	$$
and  $\dsize (A_{p+1}') ^{\frac 1{ p+1}}\ge \(\frac{(A_{p}')^{\frac 1p}}{N(h) } \)^2.
	$  From this relation \eqref{e-Ap-1} follows.
	\end{proof}

\noindent{\it Remark}. Obtaining a closed-form estimate for  the optimal frame constants  of $\G(h^{(p)}; pa, \, b/p)$  is generally not feasible, as these expressions depend
heavily on the specific properties of the function  $h$. The optimal  frame constants of  $\G(\rect^{(p)};ap, b/p)$ were explicitly evaluated in \cite{M}.

 \medskip
	The following corollary generalizes  \cite[Theorem 1.2]{DV1}.

\begin{corollary}\label{C-p-conv} Let $h$ be as in Lemma \ref{L-Ap};
 Let $\left\{\delta_{n, k}\right\} \subset  (-\frac 12, \frac 12)$    for which 
\begin{equation}\label{e-new-PW}\lambda =:4 \|h\|_1^{ p-1 } \sup_k\sum_n \|h  -h(. -  ap\delta_{n,k})\|_2^2<A_p ,	\end{equation}
	where  the $A_p$ are defined  by   \eqref{e-Ap-1}. Then, the set $\t B=\{e^{2\pi i  bn   x/p} h^{(p)}(x- ap(k+\delta_{n,k}))\}_{n, k\in\Z}$    is a frame with   frame constants
 $$  A _p'=: \(1-\sqrt{\frac{\lambda}{  A_p}}\,\)A_p, \quad   B _p'=: \(1+\sqrt{\frac{\lambda}{  A_p}}\,\)B_p.$$

\end{corollary}
 
\begin{proof} 
 We  prove that 	\begin{equation}\label{eqneq}
  	 \|h^{(p)}  -h^{(p)}(\cdot -  ap\delta_{n,k})\|_2 \leq   \|h\|_1^{ p-1 }   \|h  -h(. -  ap\delta_{n,k})\|_2^2.
\end{equation}
By  \eqref{e-new-PW} and by Theorem \ref{T-h-cmpct}  (with  $h^{(p)}$ replacing $h$ and $ap$ replacing $a$), the thesis of the corollary follows.

 When $p=1$  there is nothing to prove.  When $p>1$, we  can let $f_{p, d}(t)=h^{(p)} (t)  -h^{(p)}(t -  d) $ and observe that 
 $$
 f_{p,d}(t)= 
	 \int_{-\infty}^\infty h(s) \left(h^{(p-1)}(t -s) - h^{(p-1)}\left(t - d-s\right)\right)ds  
		 = h* f_{p-1,d}(t).
		$$
	By Young's inequality for convolution,
	$ 
	 \|f_{p,d}\|_2 \leq\|h\|_1  \|f_{p-1, d}\|_2 
	$, 
	and, if we apply this inequality $p-1$ times, we obtain $$ \|f_{p,d}\|_2  \leq\|h\|_1 ^{p-1}  \| f_{1, d}\|_2= \|h\|_1 ^{p-1} \|h   -h (\cdot -  d)\|_2.$$  
After replacing  $d$ with $ ap\delta_{n,k}$, we obtain \eqref{eqneq}. 
	
\end{proof}

\section{Other results}

 In this section we state and prove stability results   for Gabor frames $\G(a,b,h)$ for window functions   $h$  that do  not necessarily have compact support. 
  
 \subsection{ $\bf h' \ in\  W(L^\infty, \ell^1)$}
Recall that a function $g\in L^\infty(\R)$ belongs to the {\it Wiener  amalgam space}  $W(L^\infty, \ell^1)$ if
 $$
 \|g\|_{W(L^\infty, \ell^1)}=: \sum_m \|g(.+m)\|_{L^\infty(-\frac 12, \frac 12)} < \infty.$$
 
   It is  not difficult to prove that 
 $\|g\|_{W(L^\infty, \ell^1)}= \|g (.-t)\|_{W(L^\infty, \ell^1)}$ for every $t\in\R$.
Observe also that when   $g$ has compact support,   $ \|g\|_{W(L^\infty, \ell^1)}=\|g\|_{\infty}$. 
 
 \medskip
  In this section we prove the  following 
 
 \begin{theorem}\label{T-h-W}
 	Let    $\G(h; a,b)$   be a frame with frame constants $A, \, B$; assume   that  $h:\R\to\R $  is   
 	differentiable   a.e., and  $ h'\in  W(L^\infty, \ell^1)$.   If
 	$$
 \lambda=:	(ab)^{-\frac 12}\|h'\|_{W(L^\infty, \ell^1)} \(\sum_k\sup_n |\delta_{n,k}|^2\)^{\frac 12}   <\sqrt A,$$   the set
 	$\t \B=\{e^{2\pi i  bnx} h(x- a(k+\delta_{n,k}))\}_{n, k\in\Z}$ is a frame with frame bounds
 	$$A'= (1-\frac{\lambda}{\sqrt A})A, \quad B'= (1+\frac{\lambda}{\sqrt A})B.
 	$$
 \end{theorem}
  
 \medskip
  By "a.e" we mean "with the possible exception of a set of measure zero". 
  
   If    $h $  satisfies the assumptions of Theorem \ref{T-h-W} and  has compact support, then $h$ must be continuous. Indeed, the derivative of a discontinuous  function is a distribution   because a discontinuity creates a singularity in the derivative, typically represented by a Dirac delta function, 
  and so $h'$   cannot be in $L^\infty(\R)$.  
  
  If  $h$ is continuous in $\R$ and differentiable a.e. on its support we can apply both  Theorem \ref{T-h-cmpct}  and  Theorem  \ref{T-h-W}.   
 \medskip
 \begin{proof}[Proof of Theorem \ref{T-h-W}]  We  let $$f_n(x)=\sum_k a_{n,k} e^{2\pi i abn}(h(x-\mu_{n,k})-h(x-k))= \sum_k a_{n,k}e^{2\pi i abn}\int_k^{\mu_{n,k}} h'(x-t)dt.$$ 
By  Theorem \ref{T-PaleyWiener},
   the thesis of  Theorem \ref{T-h-W} follows if  we prove that 
 $ \dsize \|\sum_n f_n\|_2 \leq \lambda   
 $ 
 for every  finite sequence $\{a_{n,k}\}\subset \C$ such that 
 $\sum_{n,k} |a_{n,k}|^2=1$.
  
By duality, $\|f\|_2=\sup_{\psi\in C^\infty_0( \R)\atop{\|\psi\|_{2}=1}} \big|\int_{-\infty}^\infty f(x)\overline \psi(x)dx\big|$; 
with the triangle inequality and   the  fundamental theorem of calculus, we obtain
 \begin{equation}\label{ineq-interm} 
 \|\sum_n f_n\|_2 \leq 
 \sum_m \|\sum_n f_n\|_{L^2(m -\frac 12, \,   m+\frac 12)} =  \sum_m\sup_{\psi\in C^\infty_0( -\frac 12,  +\frac 12)\atop{\|\psi\|_{L^2( -\frac 12, \, +\frac 12)}=1}} \left| \sum_n\int_{-\infty}^\infty f_n(x)\psi (m-x)dx \right|
 $$
$$= \sum_m \sup_{\psi\in C^\infty_0( \frac 12,  \frac 12)\atop{\|\psi\|_{L^2( -\frac 12, \,  \frac 12)}=1}}   \left| \sum_{n, k} a_{n,k}  \int_k^ {\mu_{n,k}}\int_{-\infty}^\infty  e^{2\pi i abnx}  \psi  (m-x) h'(x-t)dx \, dt\right|
 \end{equation}
We let $d_j=\sup_n|\delta_{n,j}|  $  and  we observe that $(k,\  \mu_{n,k})= (k,\, k+\delta_{n,k})\subset (k,  \, k+d_k)$ 
when $\delta_{n,k}\ge 0$, and $(k,\, k+\delta_{n,k})\subset (k-d_k,  \, k )$ when $\delta_{n,k}\leq 0$.   Assuming that all $\delta_{n,k}$ are $ \ge 0$  simplifies  the notation, but will not change the substance of the proof.

Recalling the definition of short-term Fourier transform and using a  change of variables in the second integral above, we  gather
 	 $$ 
 	\|\sum_n f_n\|_{L^2(m-\frac 12, \, m+\frac 12)}=  \sup_{\psi\in C^\infty_0( -\frac 12,  \frac 12)\atop{\|\psi\|_{L^2( -\frac 12, \,  +\frac 12)}=1}}   \left| \sum_{n, k} a_{n,k}  \int_k^ {\mu_{n,k}}\int_{-\infty}^\infty  e^{2\pi i abn (x +t)}  \psi   (m-t-x ) h'(x )dx \, dt\right|
 	$$
 	$$\leq \sup_{\psi\in C^\infty_0( -\frac 12,  +\frac 12)\atop{\|\psi\|_{L^2( -\frac 12, \,  +\frac 12)}=1}}  \sum_k \int^{ k+d_k }_{  k }  \left| \sum_{n }\, a_{n,k}   \F^{\psi }(h) ( abn, m-t) dt \right|.
 	$$
Applying    the  Cauchy-Schwartz inequality  to the sum  in $n$   inside the integral  and Proposition \ref{P-Pois} gives
 	$$
 \|\sum_n f_n\|_{L^2( m-\frac 12, \,  m+\frac 12)} $$$$
 \leq  (ab)^{-\frac 12}\sup_{\psi\in C^\infty_0( -\frac 12,  \frac 12)\atop{\|\psi\|_{L^2( -\frac 12, \,   \frac 12)}=1}}    \sum_k \(\sum_{n }\,| a_{n,k} |^2\)^{\frac 12} \int^{ k +d_k}_{  k }
 \(\sum_{n }\,| F^{\psi }(h) ( abn, m-t) |^2\)^{\frac 12} dt 
 $$$$
=  (ab)^{-\frac 12}\sup_{\psi\in C^\infty_0( -\frac 12,  \frac 12)\atop{\|\psi\|_{L^2( -\frac 12, \,   \frac 12)}=1}}    \sum_k \(\sum_{n }\,| a_{n,k} |^2\)^{\frac 12}\int^{ k+d_k }_{  k } \(\int_{-\infty}^\infty   |\psi(m-t-s)|^2|h'(s)|^2 ds \)^{\frac 12} dt.
 	$$
 	 
 	After a change of variables in the second integral,
$$\|\sum_n f_n\|_{L^2( m-\frac 12, \,  m+\frac 12)}  	
\leq (ab)^{-\frac 12}\!\!\!\!\!\!\!\!\sup_{\psi\in C^\infty_0( -\frac 12,  \frac 12)\atop{\|\psi\|_{L^2( -\frac 12, \,   \frac 12)}=1}}    \sum_k \(\sum_{n }\,| a_{n,k} |^2\)^{\frac 12}\int^{ k+d_k }_{  k } \(\int_{-\infty}^\infty   |\psi(s )|^2|h'(m- t -s)|^2 ds \)^{\frac 12} dt 
$$ 
 	$$\leq
 	(ab)^{-\frac 12}\!\!\!\!\!\!\!\!\sup_{\psi\in C^\infty_0( -\frac 12,  \frac 12)\atop{\|\psi\|_{L^2( -\frac 12, \,   \frac 12)}=1}} \sum_k \(\sum_{n }\,| a_{n,k} |^2\)^{\frac 12}
 	\int^{ k+d_k }_{  k } \sup_{y\in (m-\frac 12, m+\frac 12)} |h'(y- t )|  \(\int_{-\infty}^\infty   |\psi(s )|^2 \)^{\frac 12} dt. 
 	$$
 	$$=
 	(ab)^{-\frac 12} \sum_k \(\sum_{n }\,| a_{n,k} |^2\)^{\frac 12}
 	\int^{ k+d_k }_{  k }\|h(.-t)\|_{L^\infty (m-\frac 12, m+\frac 12)}     dt.
 	$$
In view of \eqref{ineq-interm},  we gather
 	$$\sum_m\|\sum_n f_n\|_{L^2( m-\frac 12, \,  m+\frac 12)} \leq (ab)^{-\frac 12} \sum_k \(\sum_{n }\,| a_{n,k} |^2\)^{\frac 12}
 	\int^{ k+d_k }_{  k }\sum_m \|h'(.-t)\|_{L^\infty (m-\frac 12, m+\frac 12)}     dt  
 	$$
 	$$=\|h'\|_{W(L^\infty, \ell^1)} (ab)^{-\frac 12} \sum_k d_k \(\sum_{n }\,| a_{n,k} |^2\)^{\frac 12} .
 	$$
 	By  Cauchy-Schwartz
 	$$
 	\sum_m\|\sum_n f_n\|_{L^2( m-\frac 12, \,  m+\frac 12)}\leq \|h'\|_{W(L^\infty, \ell^1)} (ab)^{-\frac 12} \(\sum_k  
 	|d_k|^2\)^{\frac 12}.
 	$$
 	By assumption, $\|h'\|_{W(L^\infty, \ell^1)} (ab)^{-\frac 12} \(\sum_k  
 	|d_k|^2\)^{\frac 12} <\sqrt A$ and so 
 we have proved that  $$\|\sum_n f_n\|_2\leq \sum_m\|\sum_n f_n\|_{L^2( m-\frac 12, \,  m+\frac 12)} < \sqrt A.$$ 
  	   By  Theorem \ref{T-PaleyWiener}, the  thesis of the theorem follows.
 \end{proof}

\subsection{  $ h$  band-limited  }

In this section we assume that   $\hat h$ has compact  support.  As always, we assume that $\G(h;a, b )$ is a frame with frame constants $A,\ B$.
  We prove the following
  \begin{theorem}\label{T-hat-h-cmpct}
  	Let  $h\in L^2(\R)$  with Fourier transform   supported in   $(-\frac M2, \frac M2)$, with $M\in\N$.
 
  	Let $\{\delta_{n,k}\}\subset (-\frac 1{4M}, \frac 1{4M})$  be such that
  	$$
  	\|\hat h\|_\infty   \sqrt{\frac M {2\pi}} \(\sum_n (1-\cos(2\pi M D_n  )+ \sin(2\pi M D_n  )  )^2  \)^{\frac 12}  =\mu <\sqrt A $$ 
  	where we have let $D_n=\sup_k|\delta_{n,k}|$. Then,  the set
  	$\t \B=\{e^{2\pi i bnx} h(x-a(k+\delta_{n,k}))\}_{n, k\in\Z}$ is a frame with frame bounds
  	$$A'= (1-\frac{\mu}{\sqrt A})A, \quad B'= (1+ \frac{\mu}{\sqrt A})B.
  	$$
  \end{theorem}

\noindent
{\it Remark.}  In view of    \eqref{e- element},  we have that  $ \frac M {2\pi} (1-\cos(2\pi M D_n  )+ \sin(2\pi M D_n  )  )\ge  \frac {M^2}{\pi} D_n$, and so 
it is enough to verify that
$ 
\frac{M}{\sqrt \pi} \|\hat h\|_\infty    \(\sum_n D_n^2\)^{\frac 12}  =\mu<\sqrt A$

  \begin{proof}
  	We have proved in section $3.1$ that $\G(h; a,b)$ is a frame  with constants $A$ and $B$   if and only if  the set 
  	$$\G(h_a; 1,ab )= \{e^{2\pi i abnx} h_a(x- k)\}_{n, k\in\Z} $$ is a frame  with frame constants $a^{-1} A$ and $a^{-1}B$.
  Recall that $h_a(x)=h(ax)$,	
and  $  \mu_{n,k}$ may  be used in place of $k+\delta_{n,k}$.

Let $\{a_{n,k}\}\subset \C$ with $\sum_{n,k}|a_{n,k}|=1$, and    let 
  $\dsize \I=:\left\|\sum_{n,k}\, a_{n,k} \, e^{2\pi i   ab n x }\left(h_a \left(  \cdot-\mu_{n,k} \right)-h_a \left(  \cdot-k \right)\right)\right\|_2  $.
By  Theorem \ref{T-PaleyWiener},
the thesis of  Theorem  \ref{T-hat-h-cmpct} follows if  we prove that 
$ \I\leq \lambda  a^{-\frac 12} \sqrt A,  
$ with $\lambda<1$.

An application of  Plancherel theorem    gives
  $$
\I= 
   \big\|\sum_{n,k}\, a_{n,k} \, \hat h_a(y- abn)\(e^{2\pi i  k( abn-y)   } -e^{2\pi i  \mu_{n,k} ( abn-y)} \)\big\|_2  
  $$ 
  $$
   =
  \big\|\sum_{n,k}\, a_{n,k} \, \hat h_a(y- abn) e^{2\pi i  k( abn-y)   }(1-  e^{2\pi i  \delta_{n,k} ( abn-y)} \)\big\|_2 .
  $$ 
We apply  the triangle inequality and  the change of variables $y\to y+abn$ in the norm; letting $b_{n,k}= a_{n,k} e^{2\pi i nk ab}$  
 and   recalling that $\widehat h_a (yP)=\frac 1a \hat h(\frac ya)$, we can write the following chain of inequalities.
   $$ \I \leq \sum_{n }\big\|\widehat h_a(y- abn) \sum_k\, a_{n,k} \, e^{-2\pi i  ky   }  \(1-e^{2\pi i   \delta_{n,k} ( abn-y)} \)\big\|_2  
   $$
   $$
   =\sum_{n }\big\|\widehat h_a  \sum_k\, b_{n,k} \, e^{-2\pi i  ky   }  \(1-e^{2\pi i   \delta_{n,k} y } \)\big\|_{2} .
   $$
   $$
   = a^{-1}\sum_{n }\big\|\hat h (\frac ya) \sum_k\, b_{n,k} \, e^{-2\pi i  ky   }  \(1-e^{2\pi i   \delta_{n,k} y } \)\big\|_{L^2(-\frac{aM}2,\, \frac{aM}2)}  
   $$
   $$
 \leq  a^{-1}\| \hat h\|_\infty  \sum_{n }\big\|  \sum_k\, b_{n,k} \, e^{-2\pi i  ky   }  \(1-e^{2\pi i   \delta_{n,k} y } \)\big\|_{L^2(-\frac{aM}2,\, \frac{aM}2)} .
    $$
By Lemma \eqref{L-Gen-K},
 $$
\I  \leq a^{-1}\| \hat h\|_\infty   \sqrt{\frac{aM}{2\pi} } \sum_{n }(1- \cos(  \pi D_n aM ) +\sin(  \pi D_n aM )) \( \sum_k |b_{n, k}|^2 \)^{\frac 12}
  $$
After applying  the Cauchy-Schwartz inequality,  and observing that $|b_{n,k }|=|a_{n,k}|$, we obtain 
  $$\I \leq \|\hat h\|_\infty \sqrt{\frac{ M}{2a\pi} } \(\sum_n(1- \cos( 2\pi D_n  ) +\sin( 2\pi D_n  ))^2\)^{\frac 12} \( \sum_{n,k} |a_{n, k}|^2 \)^{\frac 12}.
  $$
 By assumption, $\sqrt{\frac{ M}{2 \pi} } \(\sum_n(1- \cos( 2\pi D_n  ) +\sin( 2\pi D_n  ))^2\)^{\frac 12}  < \sqrt{A}$ and so we have proved that $\I< a^{-\frac 12}\sqrt A$, as required
   

  \end{proof}

\medskip 
\noindent{\it Remark}. Recalling that a Gabor system  $\G=\{e^{2\pi i b_n t}   h(t-a_k)\}_{n, k\in\Z}$ is a frame if and only if 
 $\hat \G= \{e^{2\pi i a_k s} \hat h(s-b_n)\}_{n, k\in\Z}$ is a frame, we can use  Theorems \ref{T-h-cmpct}, \ref{T-hat-h-cmpct} and \ref{T-h-W}  to prove  stability results for frames 
 $\{e^{2\pi i b(n+\omega_{k,n}) t} g(t- ak)\}_{n, k\in\Z} $    
   when  $g$ or $\hat g$  have compact support,   or when    $ (\hat g)'$ is in a Wiener  amalgam space. We leave these generalizations  to the  interested reader.

 \end{document}